\documentclass[ECP,preprint]{ejpecp}

\SHORTTITLE{CLT for persistent homology of the Wiener sausage}

\TITLE{Persistent Homology of the Wiener Sausage II:\\ A Central Limit Theorem for Drifted Planar Brownian Motion}

\AUTHORS{%
  Tristan~Guillaume\\[6pt]
  \normalsize CY Cergy Paris Universit\'{e}, Laboratoire Thema, Cergy, France\\[12pt]
  \small Correspondence: Tristan Guillaume, CY Cergy Paris Universit\'{e}, Laboratoire Thema,\\
  \small 33 boulevard du port, F-95011 Cergy-Pontoise Cedex, France.\\
  \small Tel.: 33-6-12-22-45-88.\\
  \small Email: \texttt{tristan.guillaume@cyu.fr}%
}%AUTHORS


\AMSSUBJ{60F05 ; 55N31 ; 60J65 ; 60K05 ; 60D05}

\ABSTRACT{Let $X_t = B_t + \mu t$, $t \geq 0$, be planar Brownian motion with nonzero drift, and let $K_t^r = \{x \in \R^2 : \dist(x, X[0,t]) \leq r\}$ be the radius-$r$ Wiener sausage up to time~$t$. For a bounded Borel function~$\psi$ supported in a compact interval $[r_0, r_1] \subset (0,\infty)$, consider the smoothed Betti-curve functional
$\Phi_\psi(t) := \int_{r_0}^{r_1} \beta_1^t(r)\,\psi(r)\,dr$,
where $\beta_1^t(r)$ denotes the number of holes of~$K_t^r$.
In a previous paper, a regeneration scheme along the drift direction was used to prove a law of large numbers for~$\Phi_\psi(t)$. In the present paper we prove the corresponding central limit theorem. More precisely, there exist a deterministic constant~$\rho_\psi$ and a variance $\sigma_\psi^2 \geq 0$ such that
$(\Phi_\psi(t) - \rho_\psi t)/\sqrt{t} \xrightarrow{d}_{t \to \infty} \mathcal{N}(0, \sigma_\psi^2)$.
We also obtain the finite-dimensional Gaussian limit for finitely many test functions.
The proof preserves the regenerative structure of the law of large numbers, but requires a new $L^2$ analysis of the topological interface terms created at regeneration cuts. The key input is a finite-time polynomial moment bound for integrated hole counts of the Wiener sausage. This yields square-integrability of cycle increments, within-cycle oscillations, and the last incomplete-cycle remainder, which in turn allows one to combine a standard central limit theorem for stationary $1$-dependent sequences with a renewal time-change argument.}

\newcommand{\R}{\mathbb{R}}
\newcommand{\E}{\mathbb{E}}
\newcommand{\Var}{\operatorname{Var}}
\newcommand{\Cov}{\operatorname{Cov}}
\newcommand{\dist}{\operatorname{dist}}

\begin{document}
%%%%%%%%%%%%%%%%%%%%%%%%%%%%%%%%%%%%%%%%%%%%%%%%%%%%%%%%%%%%%%%%%%%

\section{Introduction}

Let
\begin{equation}\label{eq:X}
  X_t = B_t + \mu t, \quad t \geq 0,
\end{equation}
where $B$ is standard planar Brownian motion and $\mu \in \R^2 \setminus \{0\}$.
For $r > 0$ and $t \geq 0$, define the Wiener sausage
\begin{equation}\label{eq:sausage}
  K_t^r = \bigl\{ x \in \R^2 : \dist(x, X[0,t]) \leq r \bigr\}.
\end{equation}
As $r$ varies, the family $(K_t^r)_{r \geq 0}$ forms an offset filtration of compact sets.
In the planar case, its degree-one persistent homology records the creation and filling of
holes in the thickened Brownian trace.

In the first paper of this project~\cite{Guillaume2025}, a regeneration scheme along the drift
direction was used to prove a law of large numbers for smoothed degree-one persistence
observables. More precisely, for bounded Borel $\psi$ supported in a compact radius window
$[r_0, r_1] \subset (0,\infty)$, one considers
\begin{equation}\label{eq:Phipsi}
  \Phi_\psi(t) := \int_{r_0}^{r_1} \beta_1^t(r)\, \psi(r)\, dr,
\end{equation}
where $\beta_1^t(r)$ is the first Betti number of $K_t^r$.
The main result of~\cite{Guillaume2025} states that there exists a deterministic constant
$\rho_\psi$ such that
\begin{equation}\label{eq:LLN}
  \frac{\Phi_\psi(t)}{t} \to \rho_\psi \quad \text{almost surely and in } L^1.
\end{equation}
The aim of the present paper is to prove the corresponding central limit theorem.

The regenerative decomposition already contains the correct dependence structure. If $\tau_n$
denotes the regeneration times from~\cite{Guillaume2025}, then the exact increment of $\Phi_\psi$
across one regeneration cycle depends only on two consecutive recentered regeneration blocks.
The resulting cycle increments therefore form a stationary $1$-dependent sequence. What is
genuinely new in the CLT, compared with the LLN, is not the dependence structure but the
integrability scale: the interface terms created near regeneration cuts must now be controlled
in $L^2$, not just in $L^1$.

The main analytic input is therefore a finite-time polynomial moment bound for integrated hole
counts of the Wiener sausage. This estimate is crude but sufficient, because all random time
lengths arising from regeneration have exponential moments. It yields square-integrable control of:
\begin{itemize}
\item the exact cycle increment,
\item the oscillation of $\Phi_\psi$ within a cycle,
\item and the final incomplete-cycle remainder.
\end{itemize}
Once these $L^2$-bounds are available, the proof becomes probabilistic: a standard CLT applies at
regeneration times, and a renewal time-change argument transfers the limit to deterministic time.

From the point of view of topological limit theorems, the present result is the second-order
companion to the law of large numbers proved in~\cite{Guillaume2025}. For stationary random
geometric models, laws of large numbers and central limit theorems for persistence diagrams and
persistent Betti numbers have been obtained in several settings; see, for
instance,~\cite{Hiraoka2018,Krebs2022a,Krebs2022b}, as well as the survey
in~\cite{Bobrowski2018}. For some background on the Wiener sausage itself,
see~\cite{Sznitman1998}, and~\cite{Guillaume2025} for more references. The mechanism here is
different: the randomness comes from a single continuous path, and the proof uses temporal
regeneration rather than spatial ergodicity.

\subsection{Betti-curve observables}

We work with the same class of observables as in~\cite{Guillaume2025}. If $\mu_t^1$ denotes the
degree-one persistence counting measure of the offset filtration $(K_t^r)_{r \geq 0}$, then
\begin{equation}\label{eq:PhipsiPD}
  \Phi_\psi(t)
  = \int \varphi_\psi(b,d)\, \mu_t^1(db\, dd),
  \qquad
  \varphi_\psi(b,d) := \int_b^d \psi(r)\, dr.
\end{equation}
Equivalently,
\begin{equation}\label{eq:PhipsiBetti}
  \Phi_\psi(t) = \int_{r_0}^{r_1} \beta_1^t(r)\, \psi(r)\, dr.
\end{equation}
Thus the present theorem is a fluctuation theorem for the Betti-curve test class associated
with compactly supported radius weights. As emphasized already in~\cite{Guillaume2025}, this
class is natural for the regeneration method because it reduces the two-parameter persistence
problem to a one-parameter family of fixed-radius hole counts. It is not measure-determining for
the full persistence diagram, so the theorem should be understood as a CLT for smoothed Betti
observables rather than for the full persistence measure; for general background on stability and
persistence modules, see~\cite{Chazal2016,CohenSteiner2007}.

\subsection{Main results}

Our main theorem is the following.

\begin{theorem}[CLT for smoothed persistence intensity]\label{th:CLT}
  Let $\psi$ be a bounded Borel function supported in a compact interval
  $[r_0, r_1] \subset (0,\infty)$. Then there exists $\sigma_\psi^2 \geq 0$ such that
  \begin{equation}\label{eq:CLT}
    \frac{\Phi_\psi(t) - \rho_\psi t}{\sqrt{t}}
    \xrightarrow{d}_{t \to \infty}
    \mathcal{N}(0, \sigma_\psi^2).
  \end{equation}
  The variance is given by an explicit Green--Kubo formula in terms of the centered
  regeneration-cycle rewards; see Theorem~\ref{th:detCLT} below.
\end{theorem}

The same method yields a multivariate extension.

\begin{theorem}[Finite-dimensional Gaussian fluctuations]\label{th:multCLT}
  Let $\psi_1, \ldots, \psi_m$ be bounded Borel functions supported in a common compact
  interval $[r_0, r_1]$. Then
  \begin{equation}\label{eq:multCLT}
    \left(
      \frac{\Phi_{\psi_1}(t) - \rho_{\psi_1} t}{\sqrt{t}},
      \ldots,
      \frac{\Phi_{\psi_m}(t) - \rho_{\psi_m} t}{\sqrt{t}}
    \right)
    \xrightarrow{d}_{t \to \infty}
    \mathcal{N}(0, \Sigma),
  \end{equation}
  for an explicit covariance matrix $\Sigma$.
\end{theorem}

We do not address strict positivity of the limiting variance in this paper. The fluctuation
theorem is complete without it, and the variance is identified through its Green--Kubo expression.

\subsection{Organization of the paper}

Section~\ref{sec:regen} recalls the regenerative inputs from~\cite{Guillaume2025} that are
used here as black boxes. Section~\ref{sec:L2} contains the only genuinely new analytic
ingredient: a finite-time polynomial moment bound and the resulting $L^2$-control of cycle
increments, within-cycle oscillations, and the incomplete-cycle remainder.
Section~\ref{sec:CLTregen} uses these estimates to prove a central limit theorem at
regeneration times for the centered cycle rewards. Section~\ref{sec:renewal} performs the
renewal time change and proves the negligibility of the final incomplete cycle.
Section~\ref{sec:detCLT} combines these ingredients to obtain the deterministic-time central
limit theorem and the Green--Kubo formula for the limiting variance. Finally,
Section~\ref{sec:multidim} establishes the finite-dimensional Gaussian fluctuation theorem.

\section{Regenerative inputs from the LLN paper}\label{sec:regen}

In this section we isolate the structural facts imported from~\cite{Guillaume2025}. No proofs
are repeated. The role of the section is only to fix notation and identify the black-box inputs
needed in the fluctuation argument.

Let
\begin{equation}\label{eq:eU}
  e := \frac{\mu}{\|\mu\|},
  \qquad
  U_t := \langle X_t, e \rangle.
\end{equation}
Then $U$ is a one-dimensional Brownian motion with positive drift.
As in~\cite{Guillaume2025}, one constructs regeneration times by cutting the path when $U$
reaches successive levels and selecting those cuts after which the path never backtracks beyond
a fixed negative threshold. This yields stopping times
\begin{equation}\label{eq:tau}
  0 = \tau_0 < \tau_1 < \tau_2 < \cdots
\end{equation}
such that the associated recentered path blocks are i.i.d.

Write
\begin{equation}\label{eq:eta}
  \eta_n := \tau_{n+1} - \tau_n
\end{equation}
for the cycle lengths. The first key fact is that $\eta_n$ is i.i.d.\ and has exponential
moments. The second is that, for each bounded Borel $\psi$ supported in $[r_0, r_1]$, the
regeneration-time increment
\begin{equation}\label{eq:Deltan}
  \Delta_n := \Phi_\psi(\tau_{n+1}) - \Phi_\psi(\tau_n)
\end{equation}
is an exact measurable function of two consecutive recentered regeneration blocks (specifically,
blocks $n$ and $n+1$, through the interface at $\tau_{n+1}$). Consequently, $\Delta_n$ is
stationary and $1$-dependent.

The third ingredient is the interface-localization principle from~\cite{Guillaume2025}: failures
of additivity of fixed-radius hole counts across a regeneration cut are confined to a bounded
interface region around the cut. The spatial width of that region depends only on the chosen
radius window, while its temporal extent is controlled by local forward and backward window
times having exponential moments.

We summarize what will be used later.

\begin{proposition}[Imported regenerative inputs]\label{prop:imported}
  Fix a compact radius window $[r_0, r_1] \subset (0,\infty)$. Then the following hold.
  \begin{itemize}
  \item[\emph{(i)}] There exists a sequence of regeneration times $(\tau_n)_{n \geq 0}$ such
    that the associated recentered regeneration blocks are i.i.d.
  \item[\emph{(ii)}] The cycle lengths $(\eta_n)_{n \geq 0}$ are i.i.d.\ and admit
    exponential moments.
  \item[\emph{(iii)}] For every bounded Borel function $\psi$ supported in $[r_0, r_1]$,
    the regeneration-time increments
    \begin{equation}\label{eq:Deltan2}
      \Delta_n = \Phi_\psi(\tau_{n+1}) - \Phi_\psi(\tau_n)
    \end{equation}
    form a stationary $1$-dependent sequence.
  \item[\emph{(iv)}] All topological nonadditivity across a regeneration cut is localized to an
    interface region whose temporal extent is controlled by local forward and backward window
    times with exponential moments.
  \end{itemize}
\end{proposition}

The proof is a compilation of results from~\cite{Guillaume2025}.

The central limit theorem therefore reduces to one new task: upgrading the $L^1$-control of the
interface terms from~\cite{Guillaume2025} to $L^2$-control. That is the content of the next section.

\section{\texorpdfstring{$L^2$}{L2}-control of the local topological terms}\label{sec:L2}

The only new analytic step beyond~\cite{Guillaume2025} is to upgrade the local interface terms
from $L^1$ to $L^2$. Once this is done, the regenerative CLT becomes standard.

The guiding observation is the same as in~\cite{Guillaume2025}: integrated hole counts on a
compact radius window are controlled deterministically by Wiener-sausage area. Thus second
moments of the relevant topological quantities reduce to second moments of suitable local
sausage areas. Since regeneration lengths and local window lengths have exponential moments,
any polynomial deterministic-time bound is enough.

Throughout this section, $[r_0, r_1] \subset (0,\infty)$ is fixed, and constants may depend on
$r_0, r_1, \psi, \mu$, but not on time.

\subsection{Deterministic reduction to sausage area}

For a compact planar set $K$, define
\begin{equation}\label{eq:HK}
  H(K) := \int_{r_0}^{r_1} \beta_1(K^r)\, dr.
\end{equation}
The LLN paper proves the deterministic bound
\begin{equation}\label{eq:HKbound}
  H(K) \leq \frac{|K^{r_1}|}{2\pi r_0},
\end{equation}
where $|\cdot|$ denotes planar Lebesgue measure. Consequently, for every bounded Borel $\psi$
supported in $[r_0, r_1]$,
\begin{equation}\label{eq:weightedbound}
  \int_{r_0}^{r_1} \beta_1(K^r)\, |\psi(r)|\, dr
  \leq \frac{\|\psi\|_\infty}{2\pi r_0}\, |K^{r_1}|.
\end{equation}
Thus all second-moment estimates reduce to local sausage-area bounds.

\subsection{A finite-time polynomial bound}

For $t \geq 0$, let
\begin{equation}\label{eq:Ht}
  H_t := \int_{r_0}^{r_1} \beta_1^t(r)\, dr.
\end{equation}

\begin{proposition}\label{prop:polybound}
  There exists $C < \infty$ such that for every $t \geq 0$,
  \begin{equation}\label{eq:EHt2}
    \E[H_t^2] \leq C(1 + t^3).
  \end{equation}
  Equivalently, for every bounded Borel $\psi$ supported in $[r_0, r_1]$,
  \begin{equation}\label{eq:EPhipsi2}
    \E\!\left[\left(\int_{r_0}^{r_1} \beta_1^t(r)\, \psi(r)\, dr\right)^{\!2}\right]
    \leq C_\psi (1 + t^3).
  \end{equation}
\end{proposition}

\begin{proof}
  By~\eqref{eq:weightedbound}, it suffices to bound the second moment of the sausage area.
  Decompose $X$ in the orthonormal frame $(e, e^\perp)$ adapted to the drift direction:
  \begin{equation}\label{eq:Xdecomp}
    X_s = \langle X_s, e \rangle\, e + \langle X_s, e^\perp \rangle\, e^\perp.
  \end{equation}
  The longitudinal coordinate is a one-dimensional Brownian motion with positive drift, and the
  transverse coordinate is standard Brownian motion. Writing $M_t^\parallel$ and $M_t^\perp$ for
  the suprema of the centered longitudinal and transverse Brownian parts up to time $t$, the
  sausage $K_t^{r_1}$ is contained in a random rectangle whose side lengths are bounded by
  \begin{equation}\label{eq:rectangle}
    C(1 + t + M_t^\parallel), \quad C(1 + M_t^\perp).
  \end{equation}
  Standard supremum moment bounds then give
  \begin{equation}\label{eq:EKtr12}
    \E[|K_t^{r_1}|^2] \leq C(1 + t)^p
  \end{equation}
  with exponent $p = 3$. The weighted version follows from~\eqref{eq:weightedbound}.
\end{proof}

The point is not sharpness: any polynomial bound is enough because every random time horizon
used later has an exponential moment. This is the only role of Proposition~\ref{prop:polybound}.

\subsection{A random-time corollary}

\begin{corollary}\label{cor:randomtime}
  Let $\Theta$ be a nonnegative random time such that
  $\E[e^{\lambda \Theta}] < \infty$ for some $\lambda > 0$. Then
  \begin{equation}\label{eq:EHTheta2}
    \E[H_\Theta^2] < \infty,
  \end{equation}
  and, more generally,
  \begin{equation}\label{eq:EPhiTheta2}
    \E\!\left[\left(\int_{r_0}^{r_1} \beta_1^\Theta(r)\, \psi(r)\, dr\right)^{\!2}\right]
    < \infty
  \end{equation}
  for every bounded Borel $\psi$ supported in $[r_0, r_1]$.
\end{corollary}

\begin{proof}
  Exponential moments imply moments of every order, so
  Proposition~\ref{prop:polybound} immediately gives the claim.
\end{proof}

This applies in particular to regeneration lengths, sums of finitely many consecutive
regeneration lengths, and the local forward/backward windows from~\cite{Guillaume2025}.

\subsection{Cycle increments and within-cycle oscillations}

Recall the regeneration-time increment
\begin{equation}\label{eq:Deltan3}
  \Delta_n := \Phi_\psi(\tau_{n+1}) - \Phi_\psi(\tau_n).
\end{equation}
Also define the within-cycle oscillation
\begin{equation}\label{eq:Mn}
  M_n := \sup_{\tau_n \leq t \leq \tau_{n+1}} |\Phi_\psi(t) - \Phi_\psi(\tau_n)|.
\end{equation}
The LLN paper proves that both quantities are controlled pathwise by local two-block Wiener
sausages. More precisely, $|\Delta_n|$ and $M_n$ are bounded by constants times the
$r_1$-sausage area of connected path segments whose durations are dominated by
$\eta_n + \eta_{n+1}$. Since regeneration lengths have exponential moments,
Corollary~\ref{cor:randomtime} yields the $L^2$ upgrade.

\begin{proposition}\label{prop:L2cycle}
  For every bounded Borel $\psi$ supported in $[r_0, r_1]$,
  \begin{equation}\label{eq:L2DeltaMn}
    \Delta_n \in L^2, \quad M_n \in L^2.
  \end{equation}
\end{proposition}

\begin{proof}
  The pathwise bounds from~\cite[Proposition~5.10]{Guillaume2025}, combined with the
  deterministic reduction~\eqref{eq:weightedbound}, reduce both $|\Delta_n|$ and $M_n$ to the
  area of sausages built over connected path segments of random durations with exponential
  moments. Corollary~\ref{cor:randomtime} therefore applies.
\end{proof}

\subsection{The last incomplete cycle}

Define
\begin{equation}\label{eq:NtRt}
  N_t := \max\{n \geq 0 : \tau_n \leq t\},
  \qquad
  R_t := \Phi_\psi(t) - \Phi_\psi(\tau_{N_t}).
\end{equation}
Thus $R_t$ is the contribution of the last incomplete cycle.

The LLN paper proves that $|R_t|$ depends only on a local pre-cut window, a partial post-cut
segment, and their union. Using the same decomposition and Corollary~\ref{cor:randomtime}, one
obtains the $L^2$-version needed here.

\begin{proposition}\label{prop:Rt}
  One has
  \begin{equation}\label{eq:supERt2}
    \sup_{t \geq 1} \E[R_t^2] < \infty.
  \end{equation}
  Consequently,
  \begin{equation}\label{eq:RtL2}
    \frac{R_t}{\sqrt{t}} \to 0 \quad \text{in } L^2,
  \end{equation}
  hence also in probability.
\end{proposition}

\begin{proof}
  The local decomposition from~\cite[Corollary~5.11]{Guillaume2025} bounds $|R_t|$ by the sum
  of three integrated hole counts associated with connected path segments whose durations are
  controlled by the backward window, the age process, and their sum. The backward window has
  exponential moments by~\cite{Guillaume2025}, and the age process has uniformly bounded moments
  by standard renewal theory because the regeneration lengths do.
  Corollary~\ref{cor:randomtime} therefore implies the uniform $L^2$-bound.
\end{proof}

At this point all new analytic input is available: the exact cycle increments, the within-cycle
oscillations, and the final incomplete-cycle remainder are all square-integrable. The rest of
the proof is probabilistic.

\section{Central limit theorem at regeneration times}\label{sec:CLTregen}

We now assemble the cycle-level reward sequence and apply the standard CLT for stationary
$1$-dependent square-integrable sequences.

\subsection{Centered cycle rewards}

By the LLN from~\cite{Guillaume2025}, the deterministic asymptotic slope is
\begin{equation}\label{eq:rhopsi}
  \rho_\psi := \frac{\E[\Delta_0]}{\E[\eta_0]}.
\end{equation}
Define
\begin{equation}\label{eq:Yn}
  Y_n := \Delta_n - \rho_\psi \eta_n, \quad n \geq 0.
\end{equation}
The increment $\Delta_n$ is a measurable function of two consecutive regeneration blocks, while
$\eta_n$ is measurable with respect to the $n$-th regeneration block. Since the blocks are
i.i.d., the sequence $(Y_n)$ is stationary and $1$-dependent. By
Proposition~\ref{prop:L2cycle} and the exponential moments of $\eta_n$, it is square-integrable,
and by the definition of $\rho_\psi$ it is centered. Thus we have:

\begin{proposition}\label{prop:Yn}
  The sequence $(Y_n)_{n \geq 0}$ is centered, stationary, $1$-dependent, and
  square-integrable.
\end{proposition}

\begin{proof}
  By Proposition~\ref{prop:imported}(iii), the increment sequence $(\Delta_n)_{n \geq 0}$ is
  stationary and $1$-dependent, and $\eta_n$ is measurable with respect to the $n$-th
  regeneration block. Hence $(Y_n)$ is again stationary and $1$-dependent. The centering follows
  from the definition of $\rho_\psi$, and square integrability follows from
  Proposition~\ref{prop:L2cycle} together with Proposition~\ref{prop:imported}(ii).
\end{proof}

\subsection{Scalar CLT at regeneration times}

Let
\begin{equation}\label{eq:Sn}
  S_n := \sum_{k=0}^{n-1} Y_k, \quad n \geq 1, \qquad S_0 := 0.
\end{equation}
A standard central limit theorem for stationary $1$-dependent square-integrable sequences now
applies; see Hoeffding and Robbins~\cite{Hoeffding1948}, and also Bradley~\cite{Bradley2005}
for background on dependent-sequence limit theorems.

\begin{theorem}\label{th:CLTregen}
  As $n \to \infty$,
  \begin{equation}\label{eq:SnnCLT}
    \frac{S_n}{\sqrt{n}} \xrightarrow{d} \mathcal{N}(0, \sigma_{\mathrm{cyc}}^2),
  \end{equation}
  where
  \begin{equation}\label{eq:sigmacyc}
    \sigma_{\mathrm{cyc}}^2 = \Var(Y_0) + 2\Cov(Y_0, Y_1).
  \end{equation}
  Equivalently,
  \begin{equation}\label{eq:PhitaunCLT}
    \frac{\Phi_\psi(\tau_n) - \rho_\psi \tau_n}{\sqrt{n}}
    \xrightarrow{d} \mathcal{N}(0, \sigma_{\mathrm{cyc}}^2).
  \end{equation}
\end{theorem}

\begin{proof}
  Proposition~\ref{prop:Yn} gives the hypotheses of the classical central limit theorem of
  Hoeffding and Robbins~\cite{Hoeffding1948} for stationary $1$-dependent square-integrable
  sequences. Since
  \begin{equation}\label{eq:telescope}
    \Phi_\psi(\tau_n) - \rho_\psi \tau_n
    = \Phi_\psi(0) - \rho_\psi \tau_0 + S_n,
  \end{equation}
  and the initial term is deterministic, the result follows.
\end{proof}

The quantity $\sigma_{\mathrm{cyc}}^2$ is the cycle-level variance. It is the natural
Green--Kubo expression attached to the centered reward sequence.

\subsection{Functional form}

For the time-change step it is convenient to record the process-level version. Define
\begin{equation}\label{eq:Wn}
  W_n(s) := \frac{1}{\sqrt{n}}\, S_{\lfloor ns \rfloor}, \quad s \in [0,\infty).
\end{equation}

\begin{proposition}\label{prop:FCLT}
  As $n \to \infty$,
  \begin{equation}\label{eq:FCLT}
    W_n \Rightarrow \sigma_{\mathrm{cyc}}\, W
    \quad \text{in } D[0,\infty),
  \end{equation}
  where $W$ is standard Brownian motion.
\end{proposition}

\begin{proof}
  This is the functional CLT for stationary $1$-dependent square-integrable sequences
  (see~\cite[Section~19]{Billingsley1999}).
\end{proof}

Section~\ref{sec:CLTregen} is the end of the cycle-level argument. The remaining task is to
replace the deterministic cycle index $n$ by the renewal counting process $N_t$ and to show
that the final incomplete-cycle contribution is negligible at scale $\sqrt{t}$. That will be the
content of the next section.

\section{Renewal time change}\label{sec:renewal}

We now pass from the regeneration-time CLT of Section~\ref{sec:CLTregen} to deterministic time.
The argument has three ingredients: (i)~the exact decomposition at the last completed
regeneration time, (ii)~negligibility of the incomplete cycle, and (iii)~the law of large
numbers for the renewal counting process.

\subsection{Decomposition at the last completed regeneration time}

Recall the renewal counting process
\begin{equation}\label{eq:Nt}
  N_t := \max\{n \geq 0 : \tau_n \leq t\},
\end{equation}
the age process
\begin{equation}\label{eq:At}
  A_t := t - \tau_{N_t},
\end{equation}
and the incomplete-cycle remainder
\begin{equation}\label{eq:Rt2}
  R_t := \Phi_\psi(t) - \Phi_\psi(\tau_{N_t}).
\end{equation}
Then
\begin{equation}\label{eq:decomp}
  \Phi_\psi(t) - \rho_\psi t
  = \bigl(\Phi_\psi(\tau_{N_t}) - \rho_\psi \tau_{N_t}\bigr)
  + \bigl(R_t - \rho_\psi A_t\bigr).
\end{equation}
The first term is the centered fluctuation accumulated up to the last completed regeneration
cycle; the second is the error created by the final incomplete cycle.

Using the telescoping identity~\eqref{eq:telescope}, the first bracket in~\eqref{eq:decomp}
equals $S_{N_t}$ up to the deterministic initial term.

\subsection{Negligibility of the incomplete cycle}

By Proposition~\ref{prop:Rt},
\begin{equation}\label{eq:Rtneg}
  \frac{R_t}{\sqrt{t}} \to 0 \quad \text{in } L^2.
\end{equation}
It remains to control the age process on the same scale.

\begin{lemma}\label{lem:age}
  One has
  \begin{equation}\label{eq:Atneg}
    \frac{A_t}{\sqrt{t}} \to 0 \quad \text{in } L^2,
  \end{equation}
  and therefore
  \begin{equation}\label{eq:RtAtneg}
    \frac{R_t - \rho_\psi A_t}{\sqrt{t}} \to 0 \quad \text{in } L^2,
  \end{equation}
  hence also in probability.
\end{lemma}

\begin{proof}
  The regeneration lengths $\eta_n$ are i.i.d.\ and have exponential moments by
  Proposition~\ref{prop:imported}, so standard renewal theory implies that the age process has
  uniformly bounded second moment (see~\cite[Chapter~V.3]{Asmussen2003}
  or~\cite[Section~XI.3]{Feller1971}):
  \begin{equation}\label{eq:supEAt2}
    \sup_{t \geq 0} \E[A_t^2] < \infty.
  \end{equation}
  Hence~\eqref{eq:Atneg} follows immediately,
  and~\eqref{eq:RtAtneg} follows from~\eqref{eq:Rtneg} by the triangle inequality.
\end{proof}

\subsection{Renewal scaling}

Set
\begin{equation}\label{eq:m}
  m := \E[\eta_0].
\end{equation}
Since the regeneration lengths are i.i.d.\ with finite mean, the renewal law of large numbers gives
\begin{equation}\label{eq:NtLLN}
  \frac{N_t}{t} \to \frac{1}{m}
  \quad \text{almost surely and in probability.}
\end{equation}
Equivalently,
\begin{equation}\label{eq:Ntscaling}
  \frac{N_t}{t/m} \to 1 \quad \text{in probability.}
\end{equation}
Thus the random cycle index $N_t$ is asymptotically deterministic at first order.

\subsection{Random time change}

We now combine the functional CLT from Proposition~\ref{prop:FCLT} with the renewal scaling
above. For $s \geq 0$, recall
\begin{equation}\label{eq:Wn2}
  W_n(s) := \frac{1}{\sqrt{n}}\, S_{\lfloor ns \rfloor}.
\end{equation}
By Proposition~\ref{prop:FCLT},
\begin{equation}\label{eq:FCLT2}
  W_n \Rightarrow \sigma_{\mathrm{cyc}}\, W
  \quad \text{in } D[0,\infty).
\end{equation}
Since
\begin{equation}\label{eq:Ntscaling2}
  \frac{N_t}{t/m} \to 1 \quad \text{in probability,}
\end{equation}
the usual random-time-change argument (see~\cite[Theorem~14.4]{Billingsley1999},
or~\cite[Chapter~13]{Whitt2002}) yields
\begin{equation}\label{eq:SNtCLT}
  \frac{S_{N_t}}{\sqrt{t/m}}
  \xrightarrow{d}_{t \to \infty}
  \mathcal{N}(0, \sigma_{\mathrm{cyc}}^2).
\end{equation}
Equivalently,
\begin{equation}\label{eq:SNtCLT2}
  \frac{S_{N_t}}{\sqrt{t}}
  \xrightarrow{d}_{t \to \infty}
  \mathcal{N}\!\left(0, \frac{\sigma_{\mathrm{cyc}}^2}{m}\right).
\end{equation}
This is the fluctuation theorem at the last completed regeneration time.

\section{Central limit theorem at deterministic time}\label{sec:detCLT}

We now combine the decomposition~\eqref{eq:decomp}, the negligible-remainder
estimate~\eqref{eq:RtAtneg}, and the random-index limit~\eqref{eq:SNtCLT2}.

\subsection{Scalar CLT}

\begin{theorem}\label{th:detCLT}
  Let $\psi$ be a bounded Borel function supported in a compact interval
  $[r_0, r_1] \subset (0,\infty)$. Then
  \begin{equation}\label{eq:detCLT}
    \frac{\Phi_\psi(t) - \rho_\psi t}{\sqrt{t}}
    \xrightarrow{d}_{t \to \infty}
    \mathcal{N}(0, \sigma_\psi^2),
  \end{equation}
  where
  \begin{equation}\label{eq:sigmapsi}
    \sigma_\psi^2
    = \frac{1}{\E[\eta_0]}
      \bigl(\Var(Y_0) + 2\Cov(Y_0, Y_1)\bigr).
  \end{equation}
\end{theorem}

\begin{proof}
  By~\eqref{eq:decomp},
  \begin{equation}\label{eq:detCLTproof}
    \frac{\Phi_\psi(t) - \rho_\psi t}{\sqrt{t}}
    = \frac{\Phi_\psi(\tau_{N_t}) - \rho_\psi \tau_{N_t}}{\sqrt{t}}
    + \frac{R_t - \rho_\psi A_t}{\sqrt{t}}.
  \end{equation}
  The second term converges to $0$ in probability by Lemma~\ref{lem:age}. For the first term,
  \begin{equation}\label{eq:detCLTtelescope}
    \Phi_\psi(\tau_n) - \rho_\psi \tau_n
    = \Phi_\psi(0) - \rho_\psi \tau_0 + S_n,
  \end{equation}
  so the deterministic initial term is negligible after division by $\sqrt{t}$,
  and~\eqref{eq:SNtCLT2} applies. Slutsky's lemma then yields~\eqref{eq:detCLT}, with
  \begin{equation}\label{eq:sigmarel}
    \sigma_\psi^2 = \frac{\sigma_{\mathrm{cyc}}^2}{\E[\eta_0]}.
  \end{equation}
  Substituting~\eqref{eq:sigmacyc} gives~\eqref{eq:sigmapsi}.
\end{proof}

\subsection{The variance formula}

The expression
\begin{equation}\label{eq:GK}
  \sigma_\psi^2
  = \frac{1}{\E[\eta_0]}
    \bigl(\Var(Y_0) + 2\Cov(Y_0, Y_1)\bigr)
\end{equation}
is the natural Green--Kubo formula in the present regenerative setting, i.e.\ the long-run
variance of the centered reward sequence. The numerator is the asymptotic variance per
regeneration cycle of the centered reward sequence, while the denominator converts cycle index
into physical time. The appearance of only lag~$0$ and lag~$1$ covariances reflects the exact
$1$-dependence of $(Y_n)$. In particular,
\begin{equation}\label{eq:sigmanonneg}
  \sigma_\psi^2 \geq 0.
\end{equation}
Nonnegativity follows from Theorem~\ref{th:detCLT}: it is the variance of the limiting Gaussian
law. We do not address strict positivity in this note. The central limit theorem is complete
without it, and the variance is identified by~\eqref{eq:GK}.

\section{Finite-dimensional Gaussian fluctuations}\label{sec:multidim}

The scalar argument extends immediately to finitely many test functions.

Let $\psi_1, \ldots, \psi_m$ be bounded Borel functions supported in a common compact interval
$[r_0, r_1]$, and define
\begin{equation}\label{eq:Ynj}
  \Delta_n^j := \Phi_{\psi_j}(\tau_{n+1}) - \Phi_{\psi_j}(\tau_n),
  \quad
  \rho_j := \frac{\E[\Delta_0^j]}{\E[\eta_0]},
  \quad
  Y_n^j := \Delta_n^j - \rho_j \eta_n.
\end{equation}
Set
\begin{equation}\label{eq:Ynvec}
  \mathbf{Y}_n := (Y_n^1, \ldots, Y_n^m).
\end{equation}
Since each component is a measurable function of two consecutive regeneration blocks and is
square-integrable by Proposition~\ref{prop:L2cycle}, the sequence $(\mathbf{Y}_n)$ is centered,
stationary, $1$-dependent, and square-integrable as an $\R^m$-valued sequence.

For $a \in \R^m$, the scalar sequence $\langle a, \mathbf{Y}_n \rangle$ is again centered,
stationary, $1$-dependent, and square-integrable, so Theorem~\ref{th:detCLT} applies to the
corresponding linear combination. Therefore, by the Cram\'{e}r--Wold device
(see, e.g.,~\cite{Feller1971}), one obtains the multivariate Gaussian limit.

\begin{theorem}\label{th:multidim}
  Let $\psi_1, \ldots, \psi_m$ be bounded Borel functions supported in a common compact
  interval $[r_0, r_1]$. Then
  \begin{equation}\label{eq:multidimCLT}
    \left(
      \frac{\Phi_{\psi_1}(t) - \rho_1 t}{\sqrt{t}},
      \ldots,
      \frac{\Phi_{\psi_m}(t) - \rho_m t}{\sqrt{t}}
    \right)
    \xrightarrow{d}_{t \to \infty}
    \mathcal{N}(0, \Sigma),
  \end{equation}
  where the covariance matrix $\Sigma = (\Sigma_{ij})_{1 \leq i,j \leq m}$ is given by
  \begin{equation}\label{eq:Sigmaij}
    \Sigma_{ij}
    = \frac{1}{\E[\eta_0]}
      \bigl(\Cov(Y_0^i, Y_0^j) + \Cov(Y_0^i, Y_1^j) + \Cov(Y_1^i, Y_0^j)\bigr).
  \end{equation}
\end{theorem}

\begin{proof}
  Apply the scalar deterministic-time CLT to each linear combination
  \begin{equation}\label{eq:lincomb}
    \sum_{j=1}^m a_j\, \Phi_{\psi_j}(t),
  \end{equation}
  whose centered cycle reward is $\langle a, \mathbf{Y}_n \rangle$. The covariance formula
  follows by expanding the variance of the limiting scalar Gaussian and using $1$-dependence.
\end{proof}

This completes the fluctuation theory for the Betti-curve test class associated with compactly
supported radius weights.

%%%%%%%%%%%%%%%%%%%%%%%%%%%%%%%%%%%%%%%%%%%%%%%%%%%%%%%%%%%%%%%%%%%
%% Bibliography                                                  %%
%%%%%%%%%%%%%%%%%%%%%%%%%%%%%%%%%%%%%%%%%%%%%%%%%%%%%%%%%%%%%%%%%%%


\begin{thebibliography}{99}

\bibitem{Asmussen2003}
S.~Asmussen,
\emph{Applied Probability and Queues}, 2nd ed.,
Springer, New York, 2003.

\bibitem{Billingsley1999}
P.~Billingsley,
\emph{Convergence of Probability Measures}, 2nd ed.,
Wiley, New York, 1999.

\bibitem{Bobrowski2018}
O.~Bobrowski and M.~Kahle,
Topology of random geometric complexes: a survey,
\emph{J.\ Appl.\ Comput.\ Topol.}
\textbf{1} (2018), 331--364.

\bibitem{Bradley2005}
R.~C.~Bradley,
Basic properties of strong mixing conditions.\ A survey and some open questions,
\emph{Probab.\ Surv.}
\textbf{2} (2005), 107--144.

\bibitem{Chazal2016}
F.~Chazal, V.~de Silva, M.~Glisse, and S.~Oudot,
\emph{The Structure and Stability of Persistence Modules},
Springer, 2016.

\bibitem{CohenSteiner2007}
D.~Cohen-Steiner, H.~Edelsbrunner, and J.~Harer,
Stability of persistence diagrams,
\emph{Discrete Comput.\ Geom.}
\textbf{37} (2007), 103--120.

\bibitem{Feller1971}
W.~Feller,
\emph{An Introduction to Probability Theory and Its Applications, Vol.~II},
2nd ed., Wiley, New York, 1971.

\bibitem{Guillaume2025}
T.~Guillaume,
\emph{Persistence of the Wiener Sausage: Sampling Stability and a Law of Large Numbers
for Drifted Planar Brownian Motion},
\ARXIV{2604.03130}

\bibitem{Hiraoka2018}
Y.~Hiraoka, T.~Shirai, and K.~D.~Trinh,
Limit theorems for persistence diagrams,
\emph{Ann.\ Appl.\ Probab.}
\textbf{28} (2018), 2740--2780.

\bibitem{Hoeffding1948}
W.~Hoeffding and H.~Robbins,
The central limit theorem for dependent random variables,
\emph{Duke Math.\ J.}
\textbf{15} (1948), 773--780.

\bibitem{Krebs2022a}
J.~Krebs and C.~Hirsch,
Functional central limit theorems for persistent Betti numbers on cylindrical networks,
\emph{Scand.\ J.\ Stat.}
\textbf{49} (2022), 427--454.

\bibitem{Krebs2022b}
J.~T.~N.~Krebs and W.~Polonik,
On the asymptotic normality of persistent Betti numbers,
\emph{Electron.\ J.\ Stat.}
\textbf{16} (2022), 2551--2588.

\bibitem{Sznitman1998}
A.-S.~Sznitman,
\emph{Brownian Motion, Obstacles and Random Media},
Springer, Berlin, 1998.

\bibitem{Whitt2002}
W.~Whitt,
\emph{Stochastic-Process Limits},
Springer, New York, 2002.

\end{thebibliography}
\end{document}